\documentclass[12pt,leqno,a4paper,english]{amsart}
\usepackage{amssymb}  
\usepackage{amsmath}  
\usepackage{amsthm}  
\usepackage{verbatim} 
\usepackage{url}
\usepackage{babel}
\theoremstyle{plain} 
\newtheorem{theorem}{Theorem}[section]

\newtheorem{lemma}{Lemma}[section]
\newtheorem{corollary}{Corollary}[section]

\theoremstyle{definition}

\newtheorem{remark}{Remark}[section]

\theoremstyle{definition}

\newtheorem{example}{Example}[section]

\numberwithin{equation}{section}

\newcommand{\gep}{\varepsilon}

\newcommand{\gth}{\vartheta}

\newcommand{\gl}{\lambda}
\newcommand{\gL}{\Lambda}

\newcommand{\gf}{\varphi}

\newcommand{\R}{\mathbb {R}}

\newcommand{\be}{\begin{equation}}
\newcommand{\ee}{\end{equation}}

\begin{document} 

\title[A priori estimates for the $p$-Laplacian]
{A  priori estimates for some elliptic equations involving the  $p$-Laplacian}
\author{Lucio Damascelli, Rosa Pardo}
\address[L. Damascelli]{Dipartimento di Matematica,	Universit\`a di Roma Tor Vergata, 00133--Roma, Italy}
\email{damascel@mat.uniroma2.it}
\address[R. Pardo]{Departamento de Matem\'atica Aplicada,  Universidad Complutense de Madrid
	28040--Madrid, Spain}
\email{rpardo@ucm.es}
\date{}
\thanks{Rosa Pardo was partially supported by Grant MTM2016-75465, MINECO, Spain and Grupo de Investi\'on CADEDIF 920894, UCM.}
\begin{abstract}
We consider the Dirichlet problem for positive solutions of the equation
$-\Delta_p (u) = f(u)$ in a convex, bounded, smooth domain $\Omega  \subset\R^N$, with $f$ locally Lipschitz continuous. \par
We provide sufficient conditions guarantying $L^{\infty} $ a priori bounds for positive solutions of some elliptic equations involving the $p$-Laplacian and extend the class of known nonlinearities for which the solutions  are $L^{\infty} $ a priori  bounded. As a consequence we prove the existence of positive solutions in convex bounded domains.
\end{abstract}
 
\maketitle

\section{Introduction and statement of the results.}

Let  $\Omega  $ be a smooth bounded domain in $\R^N$, $N \geq 2$, which is strictly convex. 
We are interested in proving $L^{\infty}(\Omega)$ a priori bounds for 
$  C^1 (\overline {\Omega })$  weak solutions of the problem

\begin{equation}\label{MainProblem}
\begin{cases}
-\Delta_p (u) = f(u) &\qquad \text{in }\, \Omega  \\
\qquad\ \ \, u >\,0 &\qquad  \text{in }\, \Omega  \\
\qquad\ \ \, u = \, 0 &\qquad \text{on }\, \partial\Omega, 
\end{cases}
\end{equation}
\par
\smallskip
where 
$\Delta_p
(u)=$
div$(|Du|^{p-2}Du)$ is the $p$-Laplace operator, $1<p<\infty$, and 
\begin{itemize}
\item[$H_1)$  ] \ $f:[0,\infty)\rightarrow \mathbb{R}$ is a  locally Lipschitz  continuous function and  \par
 $f(0) \geq 0$.\\
If $p>2 $ we also suppose that
$f(s)>0 $ if $s>0$.
\end{itemize}
\par
\medskip
The equation $- \Delta _p u = f(u)$ is the $L^p $ counterpart to the classical semilinear elliptic equation  
$- \Delta u = f(u)$,  and appears e.g. in the theory of non-Newtonian fluids (dilatant fluids in the case $p \geq 2$, pseudo-plastic fluids in the case $1<p<2$),  see \cite{AsMa, MaPa1, MaPa2}.
\par
\smallskip
 If $u \in W^{1,p}(\Omega ) \cap L^{\infty} (\Omega )$ is  a weak solution of the problem
\eqref{MainProblem}, then $u \in C^{1,\tau}(\overline {\Omega })$ with $\tau <1$ (see \cite{DiB}, \cite{Lieb}), so that we suppose from the beginning  a $C^1$ regularity for the solution (which is anyway in general only a 
\emph{weak} solution).
\par
Moreover in the applications of a priori estimates to existence of solutions to elliptic problems, a standard setting is the  space of continuous functions, and if $u\in C^0(\overline {\Omega})$ then also $f(u)$ is continuous and the solution $u$ belongs by the cited regularity results to the space $ C^{1,\tau}(\overline {\Omega }) $.
\par
\bigskip
In the case $p=2$, i.e. when the equation under investigation is  $-\Delta (u) = f(u) $,  the problem has been widely studied. \par
After previous classical partial results (see in particular  \cite{BrTu} and the references therein), in 1981 in  two celebrated papers   Gidas and Spruck \cite{GiSp1}, \cite{GiSp2} proved a priori bounds for nonlinearities $f(u)$ that for $N \geq 3$ behave  as a subcritical power at infinity, introducing the {\it blow-up} method together with Liouville type theorems  for solutions in $\R^N$ and in the half space $\R^N_+$.\par
In the same years De Figueiredo, Lions and Nussbaum \cite{DeFLN} obtained a similar result using a different method. 
In convex domains in particular, it is based on the monotonicity results by Gidas, Ni and Nirenberg (\cite{GiNiNi}) obtained by using  the Alexan\-drov-Serrin {\it moving plane} method (\cite{Se}), (providing a priori bounds in a neighborhood of the boundary), on the Pohozaev identity (\cite{Po}) and  on the $L^p$ theory for Laplace equations given by Calderon-Zygmund estimates, see e.g. \cite{GiTr}. 
They extend  then some of the results to arbitrary smooth domains through the Kelvin transform.
\par
In recent years Castro and Pardo (\cite{CaPa}) proved, using the techniques introduced in \cite{DeFLN}, an extension of the results to the case of nonlinearities $f(u)$ more general than functions similar to subcritical powers, showing the flexibility of the method, see also  \cite{CP-ring-like, Mavinga-Pardo_JMAA_2017}. Their arguments rely on  estimating from below the radius   $R$ such that $u(x)\ge\frac12\ \|u\|_\infty=\frac12\  u(x_0)$ for any   $x\in B(x_0,R)$.\par
\medskip
For the case of the $p$-laplacian, Liouville theorems for quasilinear elliptic inequalities in $\R^N$ involving the $p$-laplacian were proved by Mitidieri and Pohozaev (\cite{MiPo1, MiPo2}), and later for more general operators by Serrin and Zou (\cite{SeZo}).\par
Using these and other methods Azizieh and Clement and Ruiz proved in very interesting papers, different versions of a priori estimates for equations of the type
\eqref{MainProblem}.\par
With the help of the blow-up procedure, Azizieh and Clement (\cite{AzCl}, see also \cite{AzClMi} for the case of systems) proved a priori estimates for the equation \eqref{MainProblem} in the case of $\Omega $ being a strictly convex domain, $1<p<2$ and  $f(u)$ growing not faster than a power $u^{\gth}$ at infinity, with 
$$
1< \gth < p_{*}-1 \; , \quad p_{*}= \frac {(N-1)p} {N-p}
$$ 
The   exponent $p_{*}= \frac {(N-1)p} {N-p}$ 
is the optimal  exponent for  Liouville theorems for elliptic \emph{inequalities}  and observe that
$$ p_{*}=  \frac {(N-1)p} {N-p}  \; < \;  p^{*} = \frac {Np}{N-p}
$$
where $p^{*} $ is the critical exponent for the Sobolev's embeddings. 
The restriction $1<p<2 $ depends on the fact that using a blow-up procedure and Liouville theorems on the whole space, they need to exclude concentration of maximum points of the solutions at the boundary, and they use some result proved in {\cite{DaPa, DaPa2} on the symmetry and monotonicity of solutions to $p$-Laplace equations in the singular case $1<p<2$,  results that were later extended to the case $p>2$ in the papers \cite{DaSc1, DaSc2}. \par
\smallskip
Ruiz (\cite{Ru}) proved, using a different technique based among other tools on Harnack type inequalities, a priori estimates for equation more general than \eqref{MainProblem}, therein  $f=f(x,u,Du)$ can  depend on $x$ and on the gradient, for any $1<p<N $ and for general domains. 
Once again the growth at infinity  with respect to $u$ must be less than powers with exponent $ \gth <{p_{*}-1}$.\par
In both papers, there is also a general discussion on how the existence of solutions follows from the a priori estimates, using some abstract results by Krasnoselskij already used in \cite{DeFLN}.
\par
\smallskip
Later, Zou (\cite{Zou}) proved Liouville theorems in half spaces that, together with the results in \cite{SeZo}, allow him to use the blow-up method and prove, in  case $1<p<N $, a priori estimates for equation more general than \eqref{MainProblem}, therein  $f=f(x,u,Du)$ can  depend on $x$ and on the gradient, and under various hypotheses on the nonlinearities. 
In particular,  it is supposed that $f=f(x,u,Du)$ grows with respect to $u$ as a subcritical power at infinity and zero.
\par
For monotonicity and Liouville type theorems in half spaces see also the papers of Farina, Montoro and Sciunzi (\cite{FaMoSc1, FaMoSc2}).
\par
In recent years related a priori estimates for general operators were established by D'Ambrosio and Mitidieri in \cite{DaMi1, DaMi2}.
\par
\bigskip
The aim of this note is to prove a priori estimates for solutions of \eqref{MainProblem} in the case of $\Omega $ being a smooth bounded strictly convex domain, for any value of $p >1$. 
In  the case $1<p<N $,  $f(u)$ is supposed to have a subcritical grow at infinity,  but allowing  more general functions than merely subcritical powers, see Example \ref{cor:apriori:ex}. \par
We use the technique introduced in \cite{DeFLN}, that  allows to give the same proof in  case $1<p<N $,  case $p=N$ and  case $p>N$.\par
Of course in the latter cases, much weaker hypotheses are needed in order to obtain the desired estimates (in particular in case $p>N $ we only need that $f(u)$ grows faster than $u^{p-1}$ at infinity,  condition that for $p=2$ is the superlinearity at infinity).\par
The proofs in \cite{DeFLN}  for the case $p=2$ rely deeply (among many other tools that we had to modify to handle in our case) on the $C^2$ regularity of the solutions, which are then classical solutions, and on the $W^{2,q}$ estimates based on the  Calderon-Zygmund estimates. \par
These estimates are not available in the singular/degenerate case of  $p \neq 2$,  and we think that   the use of regularity properties and other tools, that we exploit to implement the method,  could be interesting also for other problems.
\par
\medskip
Let us state in more detail the results that we prove.
\par
\medskip

\begin{theorem}[Case $p>N$] \label{p>N} Let  $\Omega  $ be a smooth bounded domain in $\R^N$, $N \geq 2$, which is strictly convex.\par
Suppose that $p>N$,  the condition $H_1)$ holds   and 
\begin{itemize}
\item[$H_2)$  ] \  There exist $\tau >0 $ and $C_1>0$ such that 
\par
\smallskip
$\liminf _{s \to +\infty } \frac {f(s)}{s^{p-1+ \tau }}  >C_1>0 $ , 
\par
\smallskip
\end{itemize}
Then, the solutions of \eqref{MainProblem} are a priori bounded in $L^{\infty}$:
there exists a constant $C$, depending on $p$, $\Omega  $ and $f$, but independent of the solution $u$, such that
$\Vert u  \Vert _{L^{\infty }(\Omega )} \leq C$
for any solution of \eqref{MainProblem}.
\end{theorem}
\par
\bigskip

\begin{remark} For the ordinary laplacian  ($p=2$), the above theorem corresponds to the case of dimension $N=1$, and in \cite[Remark 1.3]{DeFLN} it was observed that if $N=1$, solutions are uniformly bounded under the only hypothesis of superlinearity at infinity, which corresponds for $p = 2$ to the hypothesis $H_2)$ 
with $\tau =0$, and the bound from below strictly bigger than $\lambda_{1}$,  the first eigenvalue for the  Laplacian operator. \par
We need the slightly stronger form (with $\tau >0 $ but arbitrarily small) for technical reasons (use of the Picone's Identity for the $p$-laplacian).\par
\end{remark}
\par
\bigskip

\begin{theorem}[Case $p=N$] \label{p=N} Let  $\Omega  $ be a smooth bounded domain in $\R^N$, $N \geq 2$, which is strictly convex.\par
Let $p=N$ and suppose that $H_1)$ and $H_2)$   hold,   as well as 
\begin{itemize}  
\item[$H_3)$  ] There exists $C_2>0$ such that\\  $\liminf _{s \to +\infty } \frac {F(s)} {sf(s)}  > C_2>0 $ \par
where $F(t)= \int _0 ^t f(s) \, ds \ $ is a primitive of the function $f$.
\item[$H_4)$  ] There exists $\theta >0$  such that
\par
\smallskip
$ \limsup _{s \to +\infty } \frac {|f(s)|} {s^{\theta}}    < + \infty $ \par
(of course this is equivalent to \\ There exists $\eta >0$  such that
$\lim _{s \to +\infty } \frac {|f(s)|} {s^{\eta}}    =0 $  )
\end{itemize}
Then, the solutions of \eqref{MainProblem} are a priori bounded in $L^{\infty}$:
there exists a constant $C$, depending on $p$, $\Omega  $ and $f$, but independent of the solution $u$, such that
$\Vert u  \Vert _{L^{\infty }(\Omega )} \leq C$
for any solution of \eqref{MainProblem}.
\end{theorem}
\par
\bigskip

\begin{remark} If $p=2$ (the ordinary laplacian), the above theorem corresponds to  the case of dimension $N=2$, and in that case  solutions are uniformly bounded under the only hypotheses of  superlinearity  together with  polynomial growth at infinity,  cf. \cite[Theorem 1.1]{DeFLN}. All the functions $f$  growing polynomially at infinity  are included  in hypotheses $H_3)$, and $H_4)$.\par
Nevertheless, when $p=N$ the critical embedding is of exponential type, and those hypotheses are not  optimal (neither are the hypotheses in \cite{DeFLN}), and we think that they can be  improved. 
\end{remark}
\par
\bigskip

\begin{theorem}[Case $1<p<N$, first result] \label{1<p<N} Let  $\Omega  $ be a smooth bounded domain in $\R^N$, $N \geq 2$, which is strictly convex.\par
Let $1<p<N$,  and suppose that  $H_1)$ and $H_2)$   hold,   as well as 
\begin{itemize}
\item[$H_3')$  ] There exists $C_3>0$ such that\\  $\liminf _{s \to +\infty } \frac {p^* F(s)- s f(s)} {sf(s)}  >C_3>0 $
\par
\smallskip
\item[$H_4')$  ] \  
$\lim _{s \to +\infty } \frac {f(s)} {s^{p^*-1} }  = 0 $
\end{itemize}
where $p^*= \frac {Np}{N-p}$. \par
Then, the solutions of \eqref{MainProblem} are a priori bounded in $L^{\infty}$:
there exists a constant $C$, depending on $p$, $\Omega  $ and the function $f$, but independent of the solution $u$, such that
$\Vert u  \Vert _{L^{\infty }(\Omega )} \leq C$
for any solution of \eqref{MainProblem}.
\end{theorem}
\par
\bigskip

\begin{remark}
In case of the ordinary laplacian  ($p=2$), the above theorem corresponds to   the case of dimension $N\geq 3$, and in \cite{DeFLN} it was proved  under similar hypotheses.\par
We include this version here, because  the first three theorems share the same proof (and it corresponds to the \cite{DeFLN} hypotheses). \par
 Nevertheless, we also prove another result (see Theorem \ref{1<p<N refined} that follows),  weakening the hypotheses needed for the result (except for a further technical hypothesis, $H_5)$,  which is satisfied for a general class of nonlinearities) and  extends the class of nonlinearities allowed, including functions $f$ more general than subcritical powers, and can be seen as the counterpart for $p \neq 2$ to the results in \cite{CaPa} in case $p = 2$.
\end{remark}
\par
\bigskip

\begin{theorem}[Case $1<p<N$  second result] \label{1<p<N refined} Let  $\Omega  $ be a smooth bounded domain in $\R^N$, $N \geq 2$, which is strictly convex.\par
Let $1<p<N$,  and suppose that  $H_1)$ and $H_2)$  hold,   as well as 
\begin{itemize}
\item[$H_3$'') ] \  There exist a nonincreasing positive function
$H:[0, + \infty ) \to \R $  such that 
\par
\smallskip
$\liminf _{s \to +\infty } \frac {p^* F(s)- s f(s)} {H(s) sf(s)}   >0 $ 
\par
\smallskip
\item[$H_4$'') ] \  
$\lim _{s \to +\infty } \frac {f(s)} {s^{p^*-1}[H(s)]^{\frac p{N-p}} }  = 0 $
\par
\smallskip
\item[$H_5)$ ] \  There exist $C_4 >0 $, $ C_5 >0 $ such that
\par
\smallskip
$\liminf _{s \to +\infty } \frac {\min _{[\frac s2,s]}f} {f(s)}   \geq C_4 >0 $ \par
$\limsup _{s \to +\infty } \frac {\max _{[0,s] } f } {f(s)}   \leq C_5 $
\end{itemize}
Then, the solutions of \eqref{MainProblem} are a priori bounded in $L^{\infty}$:
there exists a constant $C$, depending on $p$, $\Omega  $ and the function $f$, but independent of the solution $u$, such that
$\Vert u  \Vert _{L^{\infty }(\Omega )} \leq C$
for any solution of \eqref{MainProblem}.
\end{theorem}
\par
\bigskip
The existence of solutions for \eqref{MainProblem}  follows from the a priori estimates,  with a further hypothesis about the behavior of the nonlinearity at zero. \par
This was   proved  in \cite{DeFLN} (with  the hypothesis $H_0$ below)  for $p=2$, using some variants of topological arguments, connected with theorems of Krasnoselskii (\cite{Kr}) and Rabinowitz (\cite{Ra})  based on degree theory.   It was  extended to the case of $p$-Laplace equations in \cite{AzCl, Ru, Zou}.\par
It also can be adapted to our hypotheses. More precisely, the following result holds.

\begin{theorem}\label{Existence}
Let us suppose that the hypotheses of one of the previous theorem hold, and assume also that \par
\medskip
\begin{itemize}
\item[$H_0)$ ]   $      \limsup  _{s \to 0^+ } \frac {f(s)}{s^{p-1 }}  < \gl _1 $
\end{itemize}
where $\gl _1 $ is the first eigenvalue for the $p$- Laplacian (see Section 2).
\par
\medskip
(since $f(0) \geq 0 $ by $H_1)$,  this hypothesis implies that  $f(0)=0$) 
\par
\medskip  
Then, there exists a positive solution of \eqref{MainProblem}.
\end{theorem}

\par
\bigskip

The paper is organized as follows. \par
In Section 2 we recall, and in some cases prove, all the auxiliary results that we need in  the sequel.\par
In Section 3 we give the proofs of  the a priori estimates stated in Theorems \ref{p>N}--\ref{1<p<N refined} and we give an example of an almost critical nonlinearity in the case $1<p<N$ that can be handled with the help of Theorem \ref{1<p<N refined}, but not with the previous theorems, nor with the blow-up method that relies on Liouville theorems with functions $f$ that  behave exactly as a subcritical power at infinity.
\par 
Finally in Section 4  we prove  Theorem  \ref{Existence}.  \par
\medskip

\section{Preliminaries}

\subsection{Strong maximum principles and Hopf's Lemma.  Monotonicity of the solutions and moving planes method.} \ 
\par
\bigskip

Let us first recall  a particular version of the Strong Maximum Principle 
and of the Hopf's Lemma  for
the $p$-laplacian (see \cite{Va} for the case of the $p$-laplacian and
\cite{PuSe}, \cite{PuSeZ}  for  general  quasilinear elliptic
operators).

\begin{theorem}(Strong Maximum Principle and Hopf's Lemma).\label{hop}
Let $\Omega$ be a domain in $\mathbb{R}^N$ and suppose that $u \in C^1(\Omega)$, $u \geqslant 0$
in $\Omega$, weakly solves
\[
-\Delta_p u+cu^q=g \geqslant 0 \quad \text{in  }\quad \Omega
\]
with $1 < p < \infty$, $q \geqslant p-1$, $c \geqslant 0$ and $g \in L^\infty_{loc}(\Omega)$. If
$u \neq 0$ then $u >0$ in $\Omega$. Moreover for any point $x_0 \in \partial \Omega$ where the
interior sphere condition is satisfied, and such that $u \in C^1(\Omega \cup \{x_0\})$ and
$u(x_0)=0$ we have that $\frac{\partial u}{\partial \nu}(x_0)>0$ for any inward directional derivative
(this means that if $y$ approaches $x_0$ in a ball $B \subseteq \Omega$ that has $x_0$ on its
boundary, then $\lim_{y \rightarrow x_0}\frac{u(y)-u(x_0)}{|y-x_0|}>0$).
\end{theorem}

Concerning weak and  strong comparison principles for the $p$-laplacian, see e.g. \cite{CuTa}, \cite{Da}, \cite{DaSc1}, \cite{DaSc2}, \cite{GuVe} \cite{PuSe}.
\par
Here we will only need the following elementary case of (weak) comparison principle.
\begin{theorem}(Weak comparison principle) \label{WeakCP}
Let $\Omega $ be a bounded domain in $\R^N$, $1<p< \infty$, and suppose that $u, v \in W^{1,p}(\Omega )$ weakly satisfy
\be  \begin{cases}
- \Delta _p u  \leq - \Delta _p v, &\text{ \ \ in \ } \Omega \\
\qquad u  \leq   v, &\text{ \ \ on \ } \partial \Omega 
\end{cases}
\ee
i.e. $(u-v)^+ \in W_0^{1,p} (\Omega ) \; , \; $ and for any $\gf \in C_c^1 (\Omega )$, $\gf \geq 0$, (and by density for any nonnegative $\gf \in W_0^{1,p} (\Omega )$)
$$ 
\int _{\Omega }  |Du|^{p-2}Du \cdot D \gf \leq   \int _{\Omega }  |Dv|^{p-2}Dv \cdot D \gf 
$$
Then, $u \leq v $ in $\Omega $.  
\end{theorem}
The proof is elementary, taking $\gf = (u-v)^+ $ as a test function.

\par
\bigskip

Next, we recall some results on  the monotonicity of solutions of the $p$-Laplace equations.
We consider the following problem
\begin{equation}\label{EQ:1}
\begin{cases}
-\Delta_p (u) = f(u), &\qquad \text{in}\, \Omega \\
\qquad\ \ \, u >\,0, &\qquad  \text{in}\, \Omega \\
\qquad\ \  \, u = \, 0,&\qquad \text{on}\, \partial\Omega,
\end{cases}
\end{equation}
where $\Omega$ is a bounded smooth domain in $\mathbb{R}^N$, $N\geqslant 2$,
$1<p<\infty$, and we
have the following hypotheses on $f$:
\begin{itemize}
\item[(*)] $f:[0,\infty)\rightarrow \mathbb{R}$ is a continuous function which is
locally Lipschitz continuous in $(0, \infty )$.
\end{itemize}
\par
\medskip
The results that we are going to recall, can be briefly rephrased saying that  all the conclusions of  Gidas, Ni, and Nirenberg's Theorem (see \cite{GiNiNi}, \cite{BeNi}) hold for the $p$-Laplacian, provided $f$ is  only locally Lipschitz continuous if $1<p<2$, and moreover $f(s)>0 $ if $s>0 $ for $p>2$
(see \cite{Da},  \cite{DaPa}, \cite{DaPa2}, \cite{DaSc1}, \cite{DaSc2}).
\par
\bigskip

To state more precisely  some known result about the monotonicity and symmetry of solutions of $p$-Laplace equations, we need some notations.\\
Let  $\nu$ be a direction in $\mathbb{R}^N$. For a real number $\lambda$ we define
\begin{equation}\label{lu:1.2}
T^{\nu}_{\lambda}=\{x \in \mathbb{R} : x\cdot \nu = \lambda\}
\end{equation}
\begin{equation}\label{lu:1.3}
\Omega^{\nu}_{\lambda}= \{x \in \Omega :x \cdot \nu < \lambda\}
\end{equation}
\begin{equation}\label{lu:1.4}
x^{\nu}_{\lambda}=R^{\nu}_{\lambda}(x)=x+2(\lambda - x \cdot \nu)\nu, \quad\quad x \in \mathbb{R}^N
\end{equation}
and
\begin{equation}\label{lu:1.5}
a(\nu)=\inf_{x\in \Omega}x\cdot\nu .
\end{equation}
If $\lambda> a(\nu)$ then $\Omega^{\nu}_{\lambda}$ is nonempty, thus we set
\begin{equation}\label{lu:1.6}
(\Omega^{\nu}_{\lambda})'= R^{\nu}_{\lambda}(\Omega^{\nu}_{\lambda}) .
\end{equation}
Following \cite{GiNiNi, Se} we observe that for $\lambda - a(\nu)$ small
then $(\Omega^{\nu}_{\lambda})'$ is contained in $\Omega$ and will remain in it, at least until
one of the following occurs:
\begin{itemize}
\item[(i)]$(\Omega^{\nu}_{\lambda})'$ becomes internally tangent to $\partial \Omega$ .
\item[(ii)]$\ T^{\nu}_{\lambda}$ is orthogonal to $\partial \Omega$ .
\end{itemize}
Let $\Lambda _1 (\nu)$ be the set of those  $\lambda >a(\nu)$ such that for each $\mu < \lambda$
none of the conditions (i) and (ii) holds and define
\begin{equation}\label{lu:1.7}
\lambda_{1}(\nu)= \sup  \, \Lambda _1 (\nu).
\end{equation}
Moreover let
\begin{equation}
\Lambda_2(\nu)=\{\lambda > a(\nu): (\Omega^{\nu}_\mu)'\subseteq  \Omega \quad
\forall\mu\in(a(\nu),\lambda]\},
\end{equation}
and
\begin{equation}
\lambda_2 (\nu)=\sup\,\Lambda_2(\nu).
\end{equation}
Since $\Omega $ is supposed to be smooth, note that neither $\Lambda_1(\nu) $ nor
$ \Lambda_2(\nu) $ are empty, and $\Lambda_1(\nu) \subseteq
\Lambda_2(\nu)$, so that $\lambda_1 (\nu) \leqslant \lambda_2 (\nu) $
(in the terminology of \cite{GiNiNi},
$\Omega ^{\nu}_{\lambda_1 (\nu)}$ and $\Omega ^{\nu}_{\lambda_2
(\nu)}$ correspond to the 'maximal cap', and the 'optimal cap' respectively).
Finally define
\begin{equation}
\Lambda_0(\nu)=\{\lambda > a(\nu): u\leqslant u^{\nu}_{\lambda}\quad
\forall\mu\in(a(\nu),\lambda]\},
\end{equation}
and
\begin{equation}
\lambda_0 (\nu)=\sup\,\Lambda_0(\nu).
\end{equation}

\begin{theorem}\label{TEO:SIMI}
Let $\Omega$ be a bounded smooth domain in $\mathbb{R}^N$, $N\geqslant
2$, $1<p<\infty$,
$f:[0,\infty)\rightarrow \mathbb{R}$  a continuous function which is
locally Lipschitz continuous  in $(0, \infty )$ and strictly positive in $(0, \infty )$ if $p>2$. Let
$u\in C^1(\overline{\Omega})$  be a weak solution of \eqref{EQ:1}.\\
For any direction
$\nu$ and for $\lambda$ in the interval  $(a(\nu),\lambda_1 (\nu)]$ we have
\begin{equation}\label{lu:1.8I}
u(x)\leqslant u(x^{\nu}_{\lambda})\quad \forall x \in \Omega ^{\nu}_{\lambda}.
\end{equation}

If $f$ is locally Lipschitz continuous in the closed interval $[0, \infty )$
then \eqref{lu:1.8I} 
holds for any $\lambda$ in the
interval  $(a(\nu),\lambda_2 (\nu))$ 
\end{theorem}

\begin{corollary} \label{COR:SIMI}
If $f$ is locally Lipschitz continuous in the closed
interval $[0, \infty )$ and strictly positive in $(0, \infty )$,
and the domain $ \Omega  $ is convex with respect to a direction $\nu  $ and  symmetric with
respect to the hyperplane $T_{0}^{\nu }= \left \{ x \in \mathbb
{R}^{N} : x
\cdot \nu  =0  \right \} $,  then $u$ is
symmetric, i.\ e.\  $u(x)=u(x_0^{\nu })  $, and nondecreasing in
the $\nu  $--direction in  $ \Omega  _0 ^{\nu } $. 
In particular if $\Omega  $ is a ball then $u$ is radially symmetric and radially decreasing.
\end{corollary}

\begin{remark} \label{nonnegative}

In this paper we assume  $f$  positive for $p>2$  only because we are going to exploit the monotonicity results stated in the previous theorem, obtained in \cite{DaPa, DaPa2, DaSc1, DaSc2}, and in these papers  the positivity of $f$ is assumed when $p>2$.\par
In any case, in all the results that  we prove we always assume that
$f$ satisfies  $H_1)$.
\end{remark}

\par
\bigskip

As a consequence of the previous theorem,  solutions are monotone increasing from the points on the boundary along directions that belong to a neighborhood of directions close to the inner boundary.  \par
As a further consequence  we have the following property, as observed in \cite{DeFLN} for $p=2$, which can be deduced by contradiction using the strict convexity of the domain and the monotonicity of the solutions provided by the previous cited papers (see \cite {AzCl} for a related geometric discussion).\par
\medskip
\begin{lemma} \label{LemmaMovingPlane}  Let $\Omega $ be a strictly convex bounded smooth domain, and define \ \ 
 $\Omega  _{\delta}= \{ x \in \Omega  : \text{ dist }(x , \partial \Omega  ) > \delta \}$ for $\delta >0 $. \par
Then
the following holds for a  weak solution 
$u \in C^1 (\overline {\Omega })$    of the problem \eqref{MainProblem}, where $f$ satisfies the condition 
$H_1)$
\par
\medskip

\begin{equation}\label{MovingPlanes}
\begin{cases}
& \exists \; \gamma, \epsilon >0  \text{,  depending only on } \Omega  \text{, such that } \\
&\forall \; x \in \Omega  \setminus \Omega  _{\epsilon}  \text{ there is a part of a cone } I_x \text{ with }   \\
&i) \quad u(\xi) \geq u(x) \; \quad\forall \; \xi  \in I_x,  \\
&ii) \quad I_x \subset \Omega  _{ \frac {\epsilon } 2} 
,   \\
& iii)  \quad \text{meas } (I_x) \geq \gamma .
\end{cases}
\end{equation}
\end{lemma}
Geometrically $I_x$ is a part of a cone $K_x$ with vertex in $x$, where all the $K_x$ are congruent to a fixed cone $K$, and 
if $x \in \Omega \setminus \Omega  _{ \frac {\epsilon } 2} $ then 
$I_x=K_x\cap  \Omega  _{ \frac {\epsilon } 2} $. \par
Let us emphasize that $\epsilon $ and $\gamma $ depend only on the geometry of the strictly convex, bounded, smooth domain  $\Omega $.
\par
\medskip

We will use this conditions to get  $L^{\infty }$ a priori bounds  in a neighborhood of the boundary, for the solutions on a strictly convex, bounded, smooth domain $\Omega $.\par

\par
\bigskip
\bigskip

\subsection{First eigenvalue and eigenfunction.} \ 
\par
Let $\Omega $ be a bounded domain and $1<p< \infty$. A  real number $\lambda $ is a (nonlinear) \emph{eigenvalue} of the $p$-laplacian,  with associated eigenfunction $u$ if $u \in W_0^{1,p}(\Omega )$, $u \not \equiv 0 $, solves the equation 
$
- \Delta _p u = \lambda \, |u|^{p-2}u \text{ in } \Omega 
$.
\par
Although the general theory of eigenvalues for the $p$-Laplacian is far from complete (see \cite{GaPer} and the survey \cite{Per}), the properties of the first eigenvalue are known and are the same as in the case $p=2$.
Namely the following result holds (see \cite{An}, \cite{Linq}, \cite{Per}).
\begin{theorem} \label{1eigenvalue}
Let us define 
\be \lambda _1 = \lambda _1 (p, \Omega )
= \inf \left\{ \int _{\Omega } |Dv|^p \, dx \, :  v \in W_0^{1,p}(\Omega) \, , \, \int _{\Omega} |v|^p \, dx =1   \right\}
\ee 
\par
Then, $\lambda _1 $ is the first eigenvalue (i.e. $\lambda _1 \leq \lambda $ for any eigenvalue $\lambda $), it is simple, i.e. there is only an eigenfunction  up to multiplication by a constant, and it is isolated. \par
Moreover a first eigenfunction does not change sign in $\Omega $ and by the strong maximum principle it is in fact either strictly positive or strictly negative in $\Omega $. \par
So we can select a unique eigenfunction 
$ \phi _1 $ such that  \par
$ \int _{\Omega} |\phi _1|^p \, dx =1 $ and $\phi _1 >0 $ in $\Omega $.  
\end{theorem}
\par
\bigskip

\subsection{Picone's identity and inequality.} \
\par  
The following extension of the Picone's identity for the $p$-Laplacian has been proved by Allegretto and Huang.

\begin{theorem}[\cite{AlHu}]
Let $v_1,v_2 \geq 0 $ be  differentiable functions in an open set $\Omega  $, with $v_2 >0$ and $p>1$.
Put \\
$L(v_1, v_2)= |\nabla v_1 |^p +(p-1) \frac {v_1^p} {v_2^p} |\nabla v_2 |^p - 
p\frac {v_1^{p-1}} {v_2^{p-1}} \nabla v_1 \cdot  |\nabla v_2 |^{p-2} \nabla v_2  $, \\
$R(v_1, v_2)= |\nabla v_1 |^p - \nabla \left( \,  \frac {v_1^p} {v_2^{p-1}}\, \right) \cdot |\nabla v_2 |^{p-2} \nabla v_2 $ . \\
Then, $ R(v_1, v_2)=L(v_1, v_2)$ and $ L(v_1, v_2) \geq 0 $ .\par
As a consequence we have
\begin{equation}\label{PiconeInequality}
\nabla \left( \,  \frac {v_1^p} {v_2^{p-1}}\, \right) \cdot |\nabla v_2 |^{p-2} \nabla v_2  \leq |\nabla v_1 |^p
\end{equation}
\end{theorem}

\par
\bigskip
\par
\bigskip

\subsection{Pohozaev's Identity for the $p$-Laplacian.} \ 
\par 
The following extension of the Pohozaev's identity for the $p$-Laplacian has been proved by Guedda and Veron.
\begin{theorem}[\cite{GuVe}]  \label{Pohozaev }
Let $u \in W_0^{1,p}(\Omega  ) \cap L^{\infty } ( \Omega ) $ be a weak solution of the problem  
\begin{equation}
\begin{cases}
-\Delta_p (u) = f(u) &\quad \text{in }\, \Omega  \\
\qquad\ \  u = \, 0 & \quad \text{on }\, \partial\Omega, 
\end{cases}
\end{equation}
where $\Omega $ is a bounded smooth domain in $\mathbb{R}^N$, $N\geqslant 2$, $p>1$ and
$f:[0,\infty)\rightarrow \mathbb{R}$ is a  continuous function. \par
Let
$F(t)= \int _0 ^t f(s) \, ds $, be a primitive of the function $f$.
Then, 
\begin{equation}\label{PohozaevIdentity}
N \int _{\Omega  } F(u) \, dx -  \frac {N-p}p \int _{\Omega  } f(u) \, u \, dx =
\frac {p-1}p
\int _ {\partial \Omega  } \left|\frac { \partial u}{\partial \nu}\right|^p (x \cdot \nu) \, ds  
\end{equation}
where $\nu $ is the unit exterior normal on $\partial \Omega  $.
\end{theorem}
\par
\bigskip

\par
\bigskip

\subsection{Gradient Regularity} \ \par
We are going to use the following result about the summability of the gradient for solutions to equations involving the $p$-Laplace operator. 

\begin{theorem}[Gradient Regularity]\label{th:RegularityGradient}
Let $\Omega  $ be a smooth bounded domain in $\R^N$, $N \geq 2$, and let 
$u \in W_0^{1,p}(\Omega ) $, $1 < p < \infty $, be a solution of the problem 
\begin{equation}\label{Cianchi-Iwaniek}
\begin{cases}
-\Delta_p (u) = g &\qquad \text{in }\, \Omega  \\
\qquad\ \  u = \, 0 &\qquad \text{on }\, \partial\Omega, 
\end{cases}
\end{equation}
with $g \in L^q(\Omega )$. \par
We suppose that 
\be \begin{cases}
1<q< \infty  & \text{ if } \quad p \geq N, \\
(p^{*})' \leq q < \infty &  \text{ if } \quad  1< p < N.
\end{cases}
\ee
Here 
$p^*= \frac {Np}{N-p}$ is the critical exponent for  Sobolev embedding, and 
$(p^{*})' = \frac {p^*}{p^* -1}= \frac {Np}{Np-N+p}$, is its conjugate exponent.
\begin{itemize}
\item [i)  ] If $ q<N$, then  \ $\ \Vert \nabla u \Vert _{L^{\frac {Nq(p-1)} {N-q}  } (\Omega )} \leq 
C \Vert g  \Vert _ {L^q (\Omega ) }^{\frac 1{(p-1)}}$
\item [ii)  ] If $q\geq N$, then  \ $\ \Vert \nabla u \Vert _{L^{\sigma  } (\Omega )} \leq 
C \Vert g  \Vert _ {L^q (\Omega ) }^{\frac 1{(p-1)}}$  \ for any $\sigma < \infty $. 
\end{itemize}
Here $C$ is a constant that depends on $p,N,q$.
\end{theorem}

\begin{remark} The exponent $(p^*)'$  is called the duality exponent, and the condition $q \geq (p^{*})'$ if $1<p<N $ guarantees by Sobolev's embeddings that  $g \in L^q(\Omega )$ belongs to the dual space 
$W^{-1, p'}(\Omega )$. If this is not the case, then we enter into the field of problems with measure data, and  other notions of solutions have been proposed (see \cite{BoGa}, \cite{Min}).
\end{remark}
\par

The previous theorem follows from different results  proved in several papers (see \cite{BoGa}, \cite{BWZ}, \cite{CiMa}, \cite{DiB}, \cite{DibMan}, 
\cite{Iw}, \cite{Min}, the survey \cite{CiMa2},  and the references therein), where also much more general situations are considered. \par
In the form that we need it is a consequence of the following

\begin{theorem}\label{Iwaniek2}
Let $\Omega  $ be a bounded (smooth) domain in $\R^N$, $1<p< \infty$ and let $u \in W_0^{1,p}(\Omega ) $ be a solution of the Dirichlet problem connected to the equation
$- \Delta _p u = \text{ div } \mathbf{F} $ in $\Omega $, \\
where $\mathbf{F} \in L^{p'}(\Omega  ; \R^N )  $, 
with $p' = \frac {p}{p-1}$. \par
If 
$\ \mathbf{F} \in L^{p'}(\Omega  ; \R^N ) \cap L^{s}(\Omega  ; \R^N )$, with $s \geq p'$, then
$\nabla u \in L^{(p-1)s}(\Omega ; \R^N ) $, and
$$\big\| |\nabla u| \big\| _{L^{(p-1)s}}^{p-1}  = \big\| |\nabla u|^{p-1}  \big\| _{L^s}  \leq  
C(s,p,N) \Vert \mathbf{F}  \Vert _{L^s} $$
\end{theorem}
\par
In \cite{Iw} the theorem is proved in the case $\Omega = \R^N $ (\cite{Iw}, Theorem 2), and generalized in \cite{DibMan} for systems in bounded domains $\Omega $ to obtain interior estimates of the gradient. \par
The global estimates provided by Theorem \ref{Iwaniek2}  are then proved in the paper \cite{BWZ} (see Theorem 1.8), where also much more irregular domains are considered. \par  
\bigskip

Let us show how Theorem \ref{th:RegularityGradient} follows from Theorem \ref{Iwaniek2}.\par
\begin{proof}[Proof of Theorem \ref{th:RegularityGradient} ]
Let  $u \in W_0^{1,p}(\Omega ) $ be a solution of  \eqref{Cianchi-Iwaniek}, with 
$g \in L^q(\Omega )$, $1< q<N$ (and $q \geq (p^*)'$ if $1<p<N $). \par
\smallskip
Then, we claim that $g= \text{ div } \mathbf{F} $ in $\Omega  $, with 
$\mathbf{F} \in L^{s}(\Omega  ; \R^N )$ with $s= q^*= \frac {Nq}{N-q} $. 
\par
In fact, if $z$ solves the Dirichlet problem for
$- \Delta z = g \in L^q (\Omega )$, then $z \in W^{2,q} (\Omega  )$ by the Calderon-Zygmund estimates (see \cite{GiTr}),  so that 
$\nabla z  
\in L^{s} (\Omega  ; \R^N ) $, $s= q^*$, and 
$g= \text{ div } \mathbf{F} $ in $\Omega  $, with 
$\mathbf{F} := -\nabla z \in L^{s}(\Omega  ; \R^N )$,   and 
$\Vert  \mathbf{F}  \Vert _{L^s} =  \Vert  \mathbf{F}  \Vert _{L^{q^*} }   \leq C  \Vert  g  \Vert _{L^q}  $
\par
\smallskip
Moreover $q^* \geq p'$. \par
Indeed, if $p \geq N $ then $p' \leq N'= \frac N{N-1}= 1^* < q^*$ for any $q>1$. \par
On the other side, if $1<p<N $, then it is straightforward to check that  condition  $q^* \geq p'$ is equivalent to our hypothesis \ 
$q \geq (p^*)'$. \par
\smallskip
Therefore by Theorem \ref{Iwaniek2} we get that $\nabla u \in L^{(p-1)s}(\Omega ; \R^N ) $ with $s= q^*$. \par
\smallskip
If $q \geq N $ then by the same method $g=  \text{ div }( \mathbf{F} )= -\text{ div } ( \nabla z ) $ in $\Omega  $, with 
$\mathbf{F} \in L^{s}(\Omega  ; \R^N )$  for any  $s>1 $, in particular for any $s \geq p'$.\par 
By exploiting Theorem \ref{Iwaniek2} in the same way, we get the result.
\end{proof}
\par
\bigskip 
\bigskip

We will need also a local $W^{1, \infty }$ result at the boundary, namely the following

\begin{theorem}\label{BoundaryEstimates}
Let $\Omega  $ be a smooth bounded domain in $\R^N$, $N \geq 2$, and let 
$u \in C^1 (\overline {\Omega})  $ be a solution of the problem 
\begin{equation}\label{LocalCianchi-Iwaniek}
\begin{cases}
-\Delta_p (u) = g &\qquad \text{in }\, \Omega  \\
\qquad\ \ u  >0 &\qquad \text{in }\, \Omega  \\
\qquad\ \ u = \, 0 &\qquad \text{on }\, \partial\Omega, 
\end{cases}
\end{equation}
with $g \in L^{(p^*)'}(\Omega)$. \par
For $\delta >0 $, let $\Omega  _{\delta}= \{ x \in \Omega  : \text{ dist }(x ,\partial \Omega  ) > \delta \}$,
and suppose that $u \; , \; g \in L^{\infty}(\Omega  \setminus \Omega  _{ \delta})$ with 
$ \Vert g  \Vert _ {L^{\infty} (\Omega  \setminus \Omega  _{\delta}) }  \leq M $,
$\Vert u  \Vert _ {L^{\infty} (\Omega  \setminus \Omega  _{\delta}) }  \leq M$. \par
Then,  there exists a constant $C>0$ only depending on $M$ and $\delta $ such that 
$\Vert \nabla u \Vert _{L^{\infty }  (\partial \Omega) }   \leq C $
\end{theorem}
\par
\medskip
Although similar estimates can be found in the literature, it was difficult for us to find the exact reference. \par
Therefore, we provide an ad-hoc proof based on comparison with the solution to a simple $p$-Laplace Dirichlet problem.
\bigskip

\begin{proof}
Let us consider the solution $v  \in W^{1,p} (\Omega) $  of the problem
\begin{equation}
\begin{cases}
-\Delta_p (v) = 1 &\qquad \text{in }\, \Omega  \\
\qquad\ \  v = \, 0 &\qquad \text{on }\, \partial \Omega   
\end{cases}
\end{equation}
By the regularity results in \cite{DiB}, \cite{Lieb} $v  \in C^1 (\overline {\Omega}) $, and by the weak and the strong maximum principle $v >0 $ in $\Omega $. \par
Since $u= |u|$ and
$|u|, |g| \leq  M $ in $ \Omega  \setminus \Omega  _{\delta} $, 
there exists  \\
$N=N_{M, \delta } >0 $  \ \ such that  \par
$ \quad 
\begin{cases}
- \Delta _p (N\, v)=N^{p-1}   \geq g = - \Delta _p (u) &\quad \text{  in  }  \Omega  \setminus \Omega  _{\delta}  \\
\qquad\qquad\qquad\   N\, v  \geq u   &\quad \text{ on }  \partial ( \Omega  \setminus \Omega  _{\delta} )
\end{cases}
$
\par
(it suffices to take \\
$N \geq M^{\frac 1{p-1}} \geq \Vert g  \Vert _ {L^{\infty} (\Omega  \setminus \Omega  _{\delta}) } ^ {\frac 1{p-1}} $  \ \  and  \ \ 
$N \geq \frac {M} {\inf _{\partial \Omega _{\delta}} v} \geq 
\frac { \Vert u  \Vert _ {L^{\infty} (\partial \Omega _{\delta}) }  } {\inf _{\partial \Omega _{\delta}} v}  $). \par
Putting  $ v_N= N_{M, \delta } \, v$ and $C= \sup _{\partial \Omega } |Dv|$,  by the weak comparison principle   we obtain that
\par
$u \leq v_N= Nv $ in $ \Omega  \setminus \Omega  _{\delta} $. 


If $x \in \partial \Omega $ and $\nu= \nu _i $ is the inner normal and 
$\alpha (t)= u (x + t \nu )$, $ \beta (t)= v_N (x + t \nu )$,
 it follows that
$\alpha (t) \leq  \beta (t)$ if $t \in [0, \delta)$, so that 
$\alpha ' (0) \leq \beta ' (0)$, i.e.
$ \frac { \partial u } {\partial \nu _i } (x) \leq  \frac { \partial v_N } {\partial \nu _i } (x)$. \par
Moreover since $u=v=0 $ on $\partial \Omega $, the size of the gradient coincides with the normal derivative, 
in the interior direction since $u$ and $v$ are positive inside $\Omega $, 
so that for any  $x \in \partial \Omega $ we have that \par
$|Du(x)|= \frac { \partial u } {\partial \nu _i }(x) \leq  \frac { \partial v_N } {\partial \nu _i }(x)=
|Dv_N(x)| = N  |Dv(x)|   \leq N \, C$
and $\big\|  |Du| \big\| _{L^{\infty}(\partial \Omega )} \leq  N_{M, \delta} \, C  =: C_{M, \delta}$.
\end{proof}

\par
\bigskip

\par
\bigskip

\section{Proofs of Theorems  \ref{p>N}-- \ref{1<p<N refined}}

\begin{proof}[Proof of Theorems \ref{p>N}-- \ref{1<p<N} ] \ \par 
\medskip

Let us start with a consequence of hypotheses $H_1)$ and $H_2)$, which are assumed in all the theorems that we are proving. \par
By hypothesis  $H_2)$ \ 
there exist $u_0 >0$ and $C_1 >0$ such that 
\be  \label{ConsequenceH2} u^{\tau } \leq   C_1 \frac {f(u)}{u^{p-1  }},\qquad\text{  if }   u \geq u_0.
\ee
In particular (if $1<p<2$, it is part of the hypothesis  $H_1)$ if $p>2$) we have that
\be  \label{DefinitaPositivita}  f(u) >0,\qquad  \text{ if } u \geq u_0. 
\ee
On the other hand, by hypothesis $H_1)$  there exists $\Lambda  \geq 0 $ such that 
\be  \label{Semipositivita}  \frac {f(u)} {u^{p-1}} \geq - \Lambda ,\qquad  \text{ for every  } u>0.
\ee
In fact if $p>2$ then $f$ is supposed to be positive in $(0, + \infty ) $ and \eqref{Semipositivita} is immediate.\par
If instead  $1<p<2$, as observed we have that $ f(u) >0 $   if  $ u \geq u_0$. \par 
Since $H_1)$, $f(0) \geq 0 $ and $f $ is  Lipschitz continuous in $[0,  u_0]$, hence    there exists $D \geq 0 $ such that \par
$f(u) \geq f(0) - Du,\quad $ if $0 \leq u < u_0 $, so that  \par
$\frac {f(u)} {u^{p-1}} \geq \frac {f(0)} {u^{p-1}} - D u^{2-p} \geq  - D u_0^{2-p} =: -\Lambda $  for every $u \in (0, u_0)$,\par
 and actually for every $u>0$ since $f(u)$ is positive in $[u_0, + \infty )$.
\par
\medskip

\emph{Step 1 - Uniform $L^1 $ estimates of   $u^{\tau}\phi ^p $ for some $\tau >0$}
\par 
\medskip

Since Picone's inequality \eqref{PiconeInequality} with $v_2=u$, $v_1= \phi _1$, the unique first eigenfunction which is positive in $\Omega $ and normalized in the $L^p$ norm, we obtain that 
\be \label{Picone1}
\int _{\Omega  } \frac {f(u)}{u^{p-1}} \phi _1 ^p \leq \lambda _1 \int _{\Omega } \phi _1 ^p = \lambda _1.
\ee
Therefore, taking into account \eqref{ConsequenceH2} and  \eqref{Semipositivita}, we have that
\begin{equation} \label{Picone3}
 \begin{split}     
 \int _{\Omega  }& u^{\tau} \phi _1 ^p \, dx   = \int _{[0 \leq u  < u_0]} u^{\tau} \phi _1 ^p \, dx+ \int _{[u  \geq u_0]}  u^{\tau} \phi _1 ^p \, dx      \\
 &\leq u_0^{\tau} \, \int _{\Omega  }  \phi _1 ^p \, dx + C_1 
 \int _{[u  \geq u_0]}  \frac {f(u)}{u^{p-1  }} \phi _1 ^p \, dx  \\
 &= u_0^{\tau} \, + C_1 
 \int _{\Omega }  \frac {f(u)}{u^{p-1  }} \phi _1 ^p \, dx - C_1
 \, \int _{[0 \leq u  < u_0]}  \frac {f(u)}{u^{p-1  }} \phi _1 ^p \, dx  
 \\
 &\leq u_0^{\tau} +  C_1  \lambda _1+ C_1\Lambda   = C,
 \end{split}  
 \end{equation}
for a constant $C$ independent of $u$.
\par
\medskip

\emph{Step 2 -  
Uniform $L^{ \infty } $ estimates near the boundary and uniform $W^{1, \infty } $ estimates at the boundary}
\par 
\medskip
Let $\phi _1 $ be the first eigenfunction, positive in $\Omega $ and normalized in the $L^p$ norm.
By \eqref{MovingPlanes} for any $ x \in \Omega  \setminus \Omega  _{\delta}$ we have  that \par
$\gamma \,  ( \, \inf _{ \Omega  _{ \frac {\delta }2}} \phi _1 ^p \, ) \, 
u(x)^ {\tau} \leq
\int _{I_x} u^{\tau}(\xi)\phi _1 ^p (\xi )\leq
\int _{\Omega} u^{\tau}\, \phi _1 ^p \, dx  \leq C  $
\par
where $C$ is the uniform constant obtained in the previous step.\par
This gives the uniform $L^{\infty}$ bounds near the boundary:
there exists a $\delta >0$ (depending on the geometry of $\Omega$ through \eqref{MovingPlanes}) and a constant $C>0 $ independent of the solution $u$,  such that 
\be \label{est:bound} 
\Vert u   \Vert _{L^{\infty }(\Omega \setminus \Omega  _{\delta })} \leq C
\ee
for any solution $u$ of \eqref{MainProblem}.\par
\medskip
Using Theorem \ref{BoundaryEstimates} we get 
\be
\left\|  \frac { \partial u}{\partial \nu} \right\| _{L^{\infty}(\partial \Omega )} \leq C 
\ee
 for a uniform constant $C$ independent of the solution.
\par
\medskip

\emph{Step 3 -  Uniform $W^{1, p } (\Omega ) $ estimates} \ 
\par 
\medskip
If $p>N $, since $N \geq 2$, by hypothesis  $H_1)$  we have that $f \geq 0$, so that also $F \geq 0$. By the Pohozaev's identity \eqref{PohozaevIdentity} we have that \par
$ \frac {p-N}p \int _{\Omega  } f(u) \, u \, dx \leq 
\frac {p-1}p
\int _ {\partial \Omega  } |\frac { \partial u}{\partial \nu}|^p (x \cdot \nu) \, ds $ and
$ |\frac { \partial u}{\partial \nu}|^p  \leq C
$ by the previous step, so that \par
$\int _{\Omega  } |\nabla u |^p \, dx = 
\int f(u)\, u \, dx \leq C  $.
\par 
\medskip
If $p=N $, by hypothesis $H_3)$  there exist $C_2 >0 $ and $s_0 >0 $ such that 
$uf(u) \leq C_2 F(u)$ if $u \geq s_0 $, so that by the Pohozaev's identity \eqref{PohozaevIdentity} (which in this case reduces to \\ 
$N \int _{\Omega  } F(u) \, dx  =
\frac {p-1}p
\int _ {\partial \Omega  } |\frac { \partial u}{\partial \nu}|^p (x \cdot \nu) \, ds  $
) \\
we have again  that \par
$\int _{\Omega  } |\nabla u |^N \, dx = 
\int f(u)\, u \, dx \leq C  $.
\par
\medskip

If $1<p<N$, again by Pohozaev's identity (which in this case can be written as \\
$p^* \int _{\Omega  } F(u) \, dx -  \int _{\Omega  } f(u) \, u \, dx =
\frac {p-1}{N-p}
\int _ {\partial \Omega  } |\frac { \partial u}{\partial \nu}|^p (x \cdot \nu) \, ds  $ ) \\
and since hypothesis $H_3')$,    there exists $C_3 >0 $ and $s_0 >0 $  such that \\
$ f(u)\, u \leq C_3 (p^* F(u)- f(u)\, u) $  if $u \geq s_0 $, so that \\
$\int _{\Omega  } |\nabla u |^p \, dx = 
\int f(u)\, u \, dx \leq C  $.
\par
\bigskip
This step ends the proof in the case $p> N $, since $W_0^{1,p}(\Omega  )$ is continuously injected in $L^{\infty}(\Omega  )$ if $p>N$. \par
\medskip
From now on we suppose that $1<p \leq N$. \par
\medskip

\emph{Step 4 -  Uniform $L^{q} (\Omega ) $ estimates for any $q < \infty $}. 
\par
\medskip
If $p=N$,  the uniform $W^{1, p } (\Omega ) $ estimate implies,  by  Sobolev's embedding,  that $u$ is uniformly bounded in $L^q (\Omega  )$ for any $q < \infty $.\par
\medskip
If $1<p<N$, we adapt to the $p$-laplacian the technique used in \cite{DeFLN} (that goes back to Brezis-Kato, see \cite{BrKa}) to get the $L^q$ estimates for any finite $q \geq 1$. \par
Testing the equation $- \Delta _p u = f(u)$ with $u^t$, $t \geq 1$, we get that \par
$t \int _{\Omega  } |\nabla u |^p u^{t-1} \, dx = \int  f(u) u^t \,dx $ . \par
Since $\big|\nabla u^{\frac {p-1+t} {p}  }\big| ^p = \big(\frac {p-1+t} {p}  \big) ^p |\nabla u| ^p u^{t-1}
$, if we put  $\alpha _t= \big(\frac {p} {(p-1+t)}   \big) ^p  $  we can write the previous equation as 
\begin{equation}
\alpha _t \, t \int _{\Omega  }  \big|\nabla u^{\frac {p-1+t} {p}  }\big| ^p =   \int _{\Omega  } f(u) u^t
\end{equation}
If $ \gep >0 $
by hypothesis $H_4')$  there exists $s_{\epsilon}$ such that \\
$|f(s)| s^t \leq \epsilon s^{p^*-1+t}$ if $s \geq s_{\epsilon}$, so that denoting by $C_t$ a uniform constant depending also on $t$, we get that \\
$ \int _{\Omega  }  \big|\nabla u^{\frac {p-1+t} {p}  }\big| ^p \leq  C_t \big(C_1+
\epsilon \int _{\Omega  } u^{p^*-1+t}   \, dx \big)  = C_t  +  \epsilon C_{t} 
\int _{\Omega  } u^{t+p-1} u^{p^*-p}  \, dx 
$ \par
By Sobolev's  inequality, and   H\" older's inequality with exponents $\frac {p^*}{p}$, $\frac {p^*}{p^*-p}$, we get that \\
$  \Big( \int _{\Omega  }   u^{\frac {p-1+t} {p} p^* } \, dx \Big) ^{\frac {p}{p^*}} \leq C   \int _{\Omega  }  \big|\nabla u^{\frac {p-1+t} {p}  }\big| ^p \leq 
C_t  +  \epsilon C_{t} 
\int _{\Omega  } u^{t+p-1} u^{p^*-p}  \, dx $ \par
$ 
\leq C_t  +  \epsilon C_{t} 
\Big(\int _{\Omega  } u^{ \frac {t+p-1}pp^*}\Big)^{\frac p {p^*}} \Big(\int _{\Omega  } u^{p^*}  \, dx \Big)^{\frac {p^*-p} {p^*}}
\leq  C_t +  \epsilon C_{t} \Big(\int _{\Omega  } u^{ \frac {t+p-1}pp^*}\Big)^{\frac p {p^*}}
$
since by Step 3 we have that $\int |\nabla u |^p $ is uniformly bounded, hence also 
$ (\int _{\Omega  } u^{p^*}  \, dx ) $ is uniformly bounded. \par
Taking $\epsilon $ small we get that \\
$\int u^{ \frac {t+p-1}pp^*} $ is uniformly bounded for any fixed $1\leq t < \infty $, so that \par
$\int u^{ q} $ is uniformly bounded for any fixed $q \geq p^* $ (and since $\Omega $ is bounded in fact for any 
$q \in [ 1, \infty )$).
\par
\medskip

\emph{Step 5 -  $L^{\infty  } (\Omega ) $ uniform estimates.}
\par 
\medskip
If $p=N$, since $u$ is uniformly bounded in $L^q (\Omega  )$  for any  fixed $q< \infty $, by hypothesis 
$H_4)$  also 
$f(u)$ is uniformly bounded in $L^q (\Omega  )$  for any  fixed $q< \infty $.\par
Taking any $ q $ with $1<q<N$, we get that $f(u)$ is uniformly bounded in $L^q (\Omega  )$ so that by the regularity results cited before (see Theorem \ref{th:RegularityGradient}),  $|\nabla u|  $ is 
uniformly bounded in $L^{(N-1)q^*} (\Omega  )$, with $(N-1)q^* >(N-1)1^* =N$. \par
This implies by Sobolev's embedding that $u$ is uniformly bounded in $L^{\infty}(\Omega  )$.\par
\medskip
If instead $1<p<N $, since $u$ is uniformly bounded in $L^q (\Omega  )$  for any  fixed $q< \infty $, by hypothesis $H_4')$  also 
$f(u)$ is uniformly bounded in $L^q (\Omega  )$  for any  fixed $q< \infty $.\par
Taking a $q$ with    $\frac Np < q <N $ (observe that $(p^*)' < \frac Np $ exactly when $1<p<N$),
by the regularity results cited before (see Theorem \ref{th:RegularityGradient}),  $|\nabla u|  $ is 
uniformly bounded in $L^{(p-1)q^*} (\Omega  )$, and \  $(p-1)q^* >N$ \  if \  $q> \frac Np$.\par
This implies again by Sobolev's embedding that $u$ is uniformly boun\-ded in $L^{\infty}(\Omega  )$.

\end{proof}

\par
\bigskip

\begin{proof}[Proof of Theorem \ref{1<p<N refined}] \ \par
Proceeding exactly as in the previous theorems we get uniform $L^{\infty } $ estimates near the boundary, see \eqref{est:bound}.
Next, instead of getting uniform estimates of $\int _{\Omega  } |\nabla u |^p \, dx = 
\int _{\Omega } uf(u) \, dx $ we get, from the condition $H_3'')$ and the Pohozaev's identity that 
$\qquad \int _{\Omega } uf(u) H(u) \, dx \leq C $ , 
where $C$ is a constant that does not depend on the solution. 
From this, since by hypothesis $H_2)$  we have that there exists $s_0 >0 $ such that $f(s) >0 $ for
$s \geq s_0$, it follows easily that
\begin{equation} \label{FirstEstimate} 
\int _{\Omega } u| f(u) | H(u) \, dx \leq C. 
\end{equation}
We observe now that by hypothesis $H_4'')$ , since $p^* -1 = \frac  {N(p-1)+p}{N-p}$, we have that \ \ 
$\lim _{s \to +\infty } \frac { |f(s)|^{\frac 1{p^*-1} } }  {s[H(s)]^{\frac p{N(p-1)+p}}  }  = 0 $ , 
so that, multiplying numerator and denominator by $|f(s)| H(s)^{\frac {N(p-1)}{N(p-1)+p}}$
we get that there exists a constant $C>0$ such that
\begin{equation} 
|f(s)|^{1+\frac 1{p^*-1} } H(s)^{\frac {N(p-1)}{N(p-1)+p}} \leq 
s |f(s)| H(s) + C.
\end{equation}
Using \eqref{FirstEstimate}  we get that
\begin{equation}\label{SecondEstimate}  
\int _{\Omega } |f(u)|^{1+\frac 1{p^*-1} } H(u)^{\frac {N(p-1)}{N(p-1)+p}} \leq  C.
\end{equation}
Consequently, for any $q>N/p,$  since  $N/p>1+\frac1{p^\star -1}$ precisely when $1<p<N$, and  $H$ is non-increasing,
\begin{eqnarray}\label{f:q}
&\displaystyle \int_{\Omega }& \left|f\big(u(x)\big)\right|^q\, dx  \le\\
&\le& \frac{\displaystyle \int_{\Omega } \left|f\big(u(x)\big)\right|^{1+\frac1{p^\star -1}} H(u)^\frac{N(p-1)}{N(p-1)+p}\ 
\left|f\big(u(x)\big)\right|^{q-1-\frac1{p^\star -1}}\, dx}{H\big(\|u\|_\infty\big)^\frac{N(p-1)}{N(p-1)+p}} \nonumber\\
&\le & C\
\frac{\left\|f\big(u(\cdot)\big)\right\|_\infty^{q-1-\frac1{p^\star -1}}}{H\big(\|u\|_\infty\big)^\frac{N(p-1)}{N(p-1)+p}},\nonumber
\end{eqnarray}

Let us restrict $q\in (N/p,N).$ From  Theorem \ref{th:RegularityGradient}, i),   we have that
\begin{align}\label{grad:reg}
\|\nabla u\|_{L^{\frac{Nq(p-1)}{N-q}} (\Omega )}&\leq     C\   \|f(u)\|_{L^q(\Omega )}^\frac1{p-1}\\
&\leq   C\
\left(\frac{\left\|f\big(u(\cdot)\big)\right\|_\infty^{\ 1-\frac1{q}-\frac1{(p^\star -1)q}}}{\Big[H\big(\|u\|_\infty\big)\Big]^\frac{N(p-1)}{[N(p-1)+p]q}}\right)^{\frac{1}{p-1}},\nonumber
\end{align}
and  note that, since $q > \frac Np$,
\begin{equation}\label{grad:reg:3}
r:=\frac{Nq(p-1)}{N-q}>N.
\end{equation}

From Morrey's Theorem, (see \cite[Theorem 9.12 and Corollary 9.14]{Brezis}),  there exists a constant $C$ only dependent on $\Omega ,$ $r$ and $N$ such that
\begin{equation}\label{Morrey:2}
|u(x_1)-u(x_2)|\le C |x_1-x_2|^{1-N/r}\|\nabla u\|_{L^{r}(\Omega )},\quad \forall x_1,x_2\in \Omega .
\end{equation}
Therefore, for all $x\in B(x_1,R)\subset\Omega $
\begin{equation}\label{Morrey:1}
|u(x)-u(x_1)|\le  C\ R^{1-\frac{N}{r}}\|\nabla u\|_{L^{r}(\Omega )}.
\end{equation}

\par
\bigskip

From now on, we shall argue by contradiction,  being the main idea that if a sequence of solutions $u_k$ is unbounded in $L^{\infty}(\Omega)$, then also the integrals
$\int _{\Omega } u_k | f(u_k) | H(u_k) \, dx $ tends to infinity. We achieve it by estimating the radius   $R_k$ such that $u_k(x)\ge\frac12\ \|u_k\|_\infty=\frac12\  u_k(x_k)$ for any   $x\in B(x_k,R_k)$.  \par
Let   $\{u_k\}_k$ be a sequence of $C^1(\overline {\Omega})$ positive solutions to \eqref{MainProblem} and assume that
\begin{equation}\label{uk:infty}
\lim_{k \to \infty} \|u_k\| = + \infty,\qquad \text {where}\quad \|u_k\|:=\|u_k\|_{\infty}.
\end{equation}
From the previous estimate near the boundary, let $C,\delta >0$ be as in  \eqref{est:bound}, see Theorem \ref{BoundaryEstimates}.
Let $x_k \in \overline{\Omega _{\delta}}$ be such that
$$u_k(x_k) = \max_{\Omega _{\delta}} u_k= \max_{\Omega } u_k.$$
By taking a subsequence if needed, we may assume that
there exists $x_0 \in \overline{\Omega _{\delta}}$ such that
\begin{equation}\label{x0}
\lim_{k\to \infty} x_k = x_0\in\overline{\Omega _{\delta}},\qquad\mbox{and}\quad d_0:=dist(x_0,\partial \Omega  )\ge\delta>0.
\end{equation}

Let us choose  $R_k$ such that $B_k=B(x_k,R_k)\subset\Omega $, and 
\begin{equation*}\label{Sobolev:77}
u_k(x)\ge
\frac12\ \|u_k\|\qquad\mbox{for any}\quad  x\in B(x_k,R_k),
\end{equation*}  
and there exists $y_k\in\partial B(x_k,R_k)$ such that
\begin{equation}\label{Sobolev:78:2}
u_k(y_k)=\frac12\ \|u_k\|.
\end{equation}

Let us denote by 
\begin{equation*}\label{1:Sobolev:7}
m_k:=\displaystyle\min_{[\|u_k\|/2,\|u_k\|]}f,\qquad M_k:=\displaystyle\max_{[0,\|u_k\|]}f.
\end{equation*}
From definition,  we obtain
\begin{equation}\label{1:Sobolev:66}
m_k\le f\big(u_k(x)\big) \quad\mbox{if }x\in B_k,\qquad \qquad  f\big(u_k(x)\big) \le M_k \quad\forall x\in\Omega .
\end{equation} 

Then, reasoning as in \eqref{f:q},   we obtain
\begin{eqnarray}\label{grad:reg:cr:2}
\int_{\Omega } \left|f\big(u_k\big)\right|^q\, dx &\le &  C\
\frac{M_k^{\ q-1-\frac1{p^\star -1}}}{H\big(\|u_k\|\big)^\frac{N(p-1)}{N(p-1)+p}}.
\end{eqnarray}

From gradient regularity for $p$-laplacian equations, see \eqref{grad:reg},  we deduce
\begin{equation}\label{grad:reg:2}
\|\nabla u_k\|_{L^{r}(\Omega )}
\le   C\ \left(
\frac{M_k^{\ 1-\frac1{q}-\frac1{(p^\star -1)q}}}{\Big[H\big(\|u_k\|\big)\Big]^\frac{N(p-1)}{[N(p-1)+p]q}}\right)^{\frac{1}{p-1}}.
\end{equation}
Therefore, from Morrey's Theorem, see \eqref{Morrey:1}, for any $x\in B(x_k, R_k)$
\begin{equation}\label{Sobolev:2}
|u_k(x)-u_k(x_k)|\le C\ (R_k)^{1-\frac{N}{r}}\
\left(\frac{M_k^{\ 1-\frac1{q}-\frac1{(p^\star -1)q}}}{\Big[H\big(\|u_k\|\big)\Big]^\frac{N(p-1)}{[N(p-1)+p]q}}\right)^{\frac{1}{p-1}}.
\end{equation}

Particularizing $x=y_k$ in the above inequality and from \eqref{Sobolev:78:2} we obtain
\begin{align}\label{Sobolev:3}
C\ (R_k)^{1-\frac{N}{r}}\
\left(\frac{M_k^{\ 1-\frac1{q}-\frac1{(p^\star -1)q}}}{\Big[H\big(\|u_k\|\big)\Big]^\frac{N(p-1)}{[N(p-1)+p]q}}\right)^{\frac{1}{p-1}}
&\ge 
\frac12 \|u_k\|,
\end{align}
which implies
\begin{equation}\label{Sobolev:4}
(R_k)^{1-\frac{N}{r}}\ge  \frac1{2C}\,
\frac{\|u_k\|\ \Big[H\big(\|u_k\|\big)\Big]^\frac{N}{[N(p-1)+p]q}}{M_k^{\ \left[1-\frac1{q}-\frac1{(p^\star -1)q}\right]\frac1{p-1}}}  ,
\end{equation}
or equivalently
\begin{equation}\label{Sobolev:5}
R_k\ge \left(\frac1{2C}\,\frac{\|u_k\|\ \Big[H\big(\|u_k\|\big)\Big]^\frac{N}{[N(p-1)+p]q}}
{M_k^{\ \left[1-\frac1{q}-\frac1{(p^\star -1)q}\right]\frac1{p-1}}}\right)^{1/\big(1-\frac{N}{r}\big)}.
\end{equation}

Consequently,  taking into account \eqref{1:Sobolev:66}, and that $H$ is non-increasing
\begin{eqnarray*}\label{1:u:fu:C:3}
\int_{B(x_k,R_k)} u_k |f(u_k)|H(u_k)\, dx &\ge&   \frac{1}{2}\|u_k\| H(\|u_k\|) m_k \ \omega\,(R_k)^N,
\end{eqnarray*}
where $\omega=\omega_N$ is the volume of the unit ball in $\R^N.$

Due to $B(x_k,R_k)\subset\Omega $ , substituting inequality \eqref{Sobolev:5},  and rearranging terms,  we obtain
\begin{eqnarray*}\label{1:u:fu:C:4}
&& {\kern -1cm }\int_\Omega  u_k |f(u_k)| H(u_k)\, dx\\ 
&&\ge  \frac{1}{2}\|u_k\| H(\|u_k\|) m_k\ \omega\
\left(\frac1{2C}\,\frac{
	\|u_k\|\ \Big[H\big(\|u_k\|\big)\Big]^\frac{N}{[N(p-1)+p]q}}
{M_k^{\ \left[1-\frac1{q}-\frac1{(p^\star -1)q}\right]\frac1{p-1}}}\right)^{\frac{N}{1-\frac{N}{r}}}
\\
&&= C\, m_k\ 
\left(\frac{\|u_k\|^{1+\frac1{N}-\frac{1}{r}}H\big(\|u_k\|\big)^{\frac1{N}-\frac{1}{r}+\frac{N}{[N(p-1)+p]q}}}
{M_k^{\ \left[1-\frac1{q}-\frac1{(p^\star -1)q}\right]\frac1{p-1}}}
\right)^{\frac{1}{\frac1{N}-\frac{1}{r}}}\\
&&= C\, \frac{m_k}{M_k} \
\left(\frac{\|u_k\|^{1+\frac1{N}-\frac{1}{r}}H\big(\|u_k\|\big)^{\frac1{N}-\frac{1}{r}+\frac{N}{[N(p-1)+p]q}}}
{M_k^{\ \left[1-\frac1{q}-\frac1{(p^\star -1)q}\right]\frac1{p-1}-\frac1{N}+\frac{1}{r}}}
\right)^{\frac{1}{\frac1{N}-\frac{1}{r}}}
\end{eqnarray*}

At this moment, let us observe that from hypothesis $H_5)$ 
\begin{equation}\label{1:u:fu:C:7}
\frac{m_k}{M_k}
\ge C,\qquad\mbox{for all}\quad k\qquad\mbox{big enough.}
\end{equation}

Hence, taking again into account hypothesis $H_5)$, 
we can assert that
\begin{eqnarray*}\label{1:u:fu:C:8}
&& {\kern -1.5cm }\int_\Omega  u_k |f(u_k)| H(u_k)\, dx \\
&\ge&
C\
\left(\frac{\|u_k\|^{1+\frac1{N}-\frac{1}{r}}\Big[H\big(\|u_k\|\big)\Big]^{\frac1{N}-\frac{1}{r}+\frac{N}{[N(p-1)+p]q}}}
{M_k^{\ \left[1-\frac1{q}-\frac1{(p^\star -1)q}\right]\frac1{p-1}-\frac1{N}+\frac{1}{r}}}\right)^{\frac{1}{\frac1{N}-\frac{1}{r}}}\\
&\ge&
C\
\left(\frac{\|u_k\|^{1+\frac1{N}-\frac{1}{r}}\Big[H\big(\|u_k\|\big)\Big]^{\frac1{N}-\frac{1}{r}+\frac{N}{[N(p-1)+p]q}}}
{\bigg[f\big(\|u_k\|\big)\bigg]^{\ \left[1-\frac1{q}-\frac1{(p^\star -1)q}\right]\frac1{p-1}-\frac1{N}+\frac{1}{r}}}\right)^{\frac{1}{\frac1{N}-\frac{1}{r}}}
\end{eqnarray*}

Let us denote by
\begin{align*}
a&=1+\frac1{N}-\frac{1}{r},
\qquad
b=\frac1{N}-\frac{1}{r}+\frac{N}{[N(p-1)+p]q},
\\
c&=\left[1-\frac1{q}-\frac1{(p^\star -1)q}\right]\frac1{p-1}-\frac1{N}+\frac{1}{r}.
\end{align*}
By substituting $r=\frac{Nq(p-1)}{N-q},$ see \eqref{grad:reg:3}, we obtain
\begin{equation*}
a=\frac{q[N(p-1)+p]-N}{qN(p-1)},
\quad
b=\frac{pq-\frac{Np}{N(p-1)+p}}{qN(p-1)},
\quad
c=\frac{q(N-p)- \frac{N}{p^\star -1}}{qN(p-1)},
\end{equation*}
and it is easy to see that $a,b,c>0,$ and
$\quad\frac{a}{c}=p^\star -1,\quad \frac{b}{c}=\frac{p}{N-p}$.

Finally, from hypothesis $H_4'')$  we deduce
\begin{eqnarray*}\label{1:u:fu:C:9}
&& {\kern -1.5cm }\int_\Omega  u_k |f(u_k)| H(u_k)\, dx \\
&\ge&
C\
\left(\frac{\|u_k\|^{p^*-1}\Big[H\big(\|u_k\|\big)\Big]^\frac{p}{N-p }}
{f\big(\|u_k\|\big)}\right)^{\frac{\frac{q(N-p)- \frac{N}{p^\star -1}}{qN(p-1)}}{\frac1{N}-\frac{1}{r}}}
\to  \infty\quad \text{as } k\to \infty
\end{eqnarray*}
which contradicts \eqref{FirstEstimate}, ending  the proof.
\end{proof}

We end this section by giving an example in the case $1<p<N$ of an almost critical nonlinearity $f$ that cannot be handled with the help of Theorem \ref{1<p<N}, neither with the blow-up methods that rely on the exact behavior of $f$ at infinity as a subcritical power, neither with the methods of \cite{DeFLN} (for $p=2$),
but that fulfills the hypotheses of Theorem \ref{1<p<N refined}.
\begin{example}\label{cor:apriori:ex}
Assume that $\Omega \subset \R ^N $ is a smooth bounded convex domain, $1<p<N$, and   
$u>0$ is a   $C^1 (\overline {\Omega})$ solution to
\begin{equation}
	\label{ex:elliptic:problem}
	\left\{ \begin{array}{rcll} -\Delta_p u
		&=&\dfrac{u^{p^\star -1}}{[\ln(e+u)]^\alpha}, & \qquad \mbox{in } \Omega, \\ u&=& 0, & \qquad \mbox{on } \partial \Omega,
	\end{array}\right.
\end{equation}
with $\alpha> p/(N-p )$.
\par
Then, there exists a uniform constant $C ,$ depending only on $\Omega$ and $f$ but not on the solution,  such that   
$\quad\|u\|_{L^\infty (\Omega)}\leq C.
$
\end{example}

\begin{proof} We will prove that
$f(s)  = s^{p^\star -1}/\ln(e+s)^\alpha$ with $\alpha> p/(N-p )$ satisfies our hypotheses  for  $H(s)=1/\ln(e+s)$. Hypotheses  $H_1)$-$H_2)$ and $H_5)$ hold trivially. Let us prove $H_3'') $ and $H_4'') $.
\par
\smallskip

\emph{$H_3'')$} \ \ \ \  From definition, and integrating by parts
	\begin{eqnarray}\label{eq:lim:Fsf0}
		F(t) &=& \int_0^t\dfrac{s^{p^\star -1}}{[\ln(e+s)]^\alpha}\, ds  \nonumber\\
		&=&  \frac1{p^\star }\, \dfrac{t^{p^\star }}{[\ln(e+t)]^\alpha}
		+\frac{\alpha}{p^\star }\int_0^t \left(\dfrac{1}{\ln(e+s)}\right)^{\alpha+1}\,
		\dfrac{s^{p^\star }}{e+s}\, ds
	\end{eqnarray}
Therefore, using also l'H\^{o}pital rule, and  simplifying we can write
	\begin{eqnarray*}\label{eq:lim:Fsf}
{\kern -.5cm}\lim_{t \to +\infty}&& {\kern -.5cm}
	 \frac{p^\star  F(t)-tf(t)}{tf(t)H(t)}
		= \lim_{t \to +\infty}\frac{\alpha\displaystyle\int_0^t \left(\dfrac{1}{\ln(e+s)}\right)^{\alpha+1}\,
			\dfrac{s^{p^\star }}{e+s}\, ds}{\dfrac{t^{p^\star }}{[\ln(e+t)]^{\alpha+1}}}\nonumber
		\\[.2cm]
		&&{\kern -2cm}= \lim_{ t\to +\infty}\frac{\alpha\left(\dfrac{1}{\ln(e+t)}\right)^{\alpha+1}\,
			\dfrac{t^{p^\star }}{e+t}}{p^\star \dfrac{t^{p^\star -1}}{[\ln(e+t)]^{\alpha+1}}
			-(1+\alpha)\left(\dfrac{1}{\ln(e+t)}\right)^{\alpha+2}\,
			\dfrac{t^{p^\star }}{e+t}} 
		=\frac{\alpha}{p^\star }>0 ,
	\end{eqnarray*}
and so $H_3'')$  holds.
	\par
\smallskip
	
	\emph{$H_4'')$} \ \ \   From definition of $f$
	and $H$, for any $\alpha>\dfrac{p}{(N-p )},$
	$$  \displaystyle\lim_{t \to +\infty}
	\frac{f(t)}{t^{p^\star -1}\ \big[H(t)\big]^{\frac{p}{N-p }}}
	=\lim_{t \to +\infty}\big[\ln(e+t)\big]^{\frac{p}{N-p }-\alpha}=0.$$
	 
\end{proof}
\par
\bigskip

\section{Proof of Theorem \ref{Existence}}
Here we prove Theorem \ref{Existence}, adapting to our case (and in some point simplifying) the proofs given for $p=2$ in \cite{DeFLN}  and for $1<p<N $ in \cite{AzCl, Ru, Zou}, and referring to this papers for other applications (see e.g.  \cite{AzCl} and \cite{DeFLN}  for the proof of the existence of a continuum of solutions to some related parameter problem). \par
Throughout the section we suppose that $f$ satisfies the hypothesis $H_0)$, as well as the hypotheses of one of the previous 
Theorems \ref{p>N}---\ref{1<p<N refined}. \par
\bigskip
The connection between a priori estimates and existence theorems relies on some topological argument using degree theory. \par
The abstract result, which goes back to Krasnoselskii \cite{Kr} and Amann \cite{Amann76}, and has been adapted by  De Figueiredo et al. and other authors, can be given in several formulation. We will state it as in \cite{Ru} as follows.
\begin{theorem}\label{Kras}
Let $C $ a cone in the Banach space $X$ and $K: C \to C $ continuous and compact, with $K(0)=0$.
Suppose that there exist $r>0 $, $R>r$ and a compact homotopy $H: [0,1] \times C \to C $ such that \par
$H(0,u)=K(u)$ for any $u \in C $ \ \ and
\begin{itemize}
\item[a) ] $u \neq s K(u) $ for any $s \in [0,1]$, $u \in C $ with $\Vert u \Vert =r $
\item[b) ]  $ H(t,u) \neq u $ for any $t \in [0,1]$, $u \in C $ with $\Vert u \Vert =R $ 
\item[c) ]  $ H(1,u) \neq u $ for any   $u \in C $ with $\Vert u \Vert  \leq R $ 
\end{itemize}
If $D= \{ u \in C : r < \Vert u \Vert < R   \}$
then $K $ has a fixed point in $D$	
\end{theorem}
In fact, for the proof only the basic properties of topological degree are needed (it is easy to see that if $i_C$ is the topological index we have  that $i_C (K, B_R )=0 $ by c) and b),  while 
$i_C (K, B_r )=1  $ by a), so that $i_C (K, D )=- 1  $ and $K$ has a fixed point in $D$).	
\par
\medskip
Before proving Theorem \ref{Existence}, let us first make some remarks.\par
Let $\gl \in [0, + \infty ) $ and consider the problem

\begin{equation}\label{MainProblemLambda}
\begin{cases}
-\Delta_p (u) = f(u) + \gl  &\qquad \text{in }\, \Omega  \\
\qquad\ \ \, u >\,0 &\qquad  \text{in }\, \Omega  \\
\qquad\ \ \, u = \, 0 &\qquad \text{on }\, \partial\Omega. 
\end{cases}
\end{equation}

Let us observe that it is easy to see from the proofs in Section 3 that not only
$f_{\gl }$ fulfills the same hypotheses of $f$, but also for any fixed $\lambda _0 >0 $  there exists a constant $C_{\lambda _0} >0$ such that
\be \label{costantepertuttilambda} 
\Vert u_{\gl }  \Vert _{L^{\infty }(\Omega )} \leq C _{\lambda _0}
\ee
for any solution $u_{\gl }$ of \eqref{MainProblemLambda} and
 \emph{for any $\gl  \in [0, \lambda _0 ] $}.
\par
\medskip
Moreover the proofs  in  Section 3 show that   under the hypotheses of one among Theorems \ref{p>N}---\ref{1<p<N refined},  there exists an $\epsilon>0$ (depending on the geometry of $\Omega$ through \eqref{MovingPlanes}) and a constant $C>0 $ (the  a priori bound in a neighborhood of the boundary obtained in the proof of those theorems, see \eqref{est:bound} ) such that 
\be \label{costanteuniversale} 
\Vert u_{\gl }  \Vert _{L^{\infty }(\Omega \setminus \Omega  _{\epsilon})} \leq C
\ee
for any solution $u_{\gl }$ of \eqref{MainProblemLambda} and \emph{for any $\gl  \in [0, + \infty ) $}.\par
In other words the constants $C$ obtained in Step 1 and in \eqref{est:bound} in the proof of Theorems \ref{p>N}---\ref{1<p<N refined} is independent of $\gl  \in (0, + \infty) $. \par
Analyzing in particular Step 1 of the proof, where the Picone's identity is exploited,   this assertion can be seen.\par
The  same identity shows then the following property. 

\begin{lemma}\label{costantelambda0}
There exists $\gl _0 >0 $ such that the problem
\be  \label{eqlemma}
\begin{cases}
-\Delta_p (u)   = f(u) +  \gl      &\qquad \text{in }\, \Omega  \\
\qquad\ \ \, u >\,0 &\qquad  \text{in }\, \Omega  \\
\qquad\ \ \, u = \, 0 &\qquad \text{on }\, \partial\Omega, 
\end{cases}
\ee
has no solutions if $\gl \geq \gl _0 $.
\end{lemma}

\begin{proof}
As in Step 1 of the proof in Section 3, we have by Picone's identity that \par 
$\int _{\Omega  } \frac {f(u)+ \gl }{u^{p-1}} \phi _1 ^p \leq \lambda _1 \int _{\Omega } \phi ^p= \lambda _1$, and
by \eqref{Semipositivita} we have that 
$\frac {f(u)}{u^{p-1}}$ is bounded from below by a constant $-\gL$. So we get that \par
$\int _{\Omega  } \frac {  \gl }{u^{p-1}} \phi _1 ^p \leq (\Lambda + \lambda _1) \int _{\Omega } \phi ^p=
(\Lambda + \lambda _1)$.
\par
As a consequence, using \eqref{costanteuniversale} we have that \par
$\frac {\gl} {C^{p-1}} \int _{\Omega \setminus \Omega  _{\epsilon}}  \phi _1 ^p \leq 
\int _{\Omega \setminus \Omega  _{\epsilon}  } \frac { \gl }{u^{p-1}} \phi _1 ^p \leq
\int _{\Omega   } \frac { \gl }{u^{p-1}} \phi _1 ^p \leq
 (\Lambda + \lambda _1) $, \par
where $C$ is the constant in \eqref{costanteuniversale}, so that we get the bound $\lambda _0 $ for the existence of a solution.  \par
 \end{proof}

\begin{proof}[Proof of Theorem \ref{Existence}] \ \par
To prove Theorem \ref{Existence} we will apply  Theorem \ref{Kras}.\par
\smallskip
We consider the Banach space  $X= C^0(\overline {\Omega})$, and the cone 
$C$ of nonnegative functions.\par	
\smallskip
Let us observe that in previous papers the Banach space used is the space $C^1 (\overline {\Omega})$, but it is the same (and simpler) dealing with the space of continuous functions.\par
\smallskip
If $u \in X= C^0(\overline {\Omega}) $, $\gl _0 >0 $ is the number provided in Lemma \ref{costantelambda0}, and $\Lambda $ is the number in \eqref{Semipositivita},
let  us define $v:=H(t,u)$ as the solution in $ W_0^{1,p}(\Omega)
$  of the problem 
\be
\begin{cases}
-\Delta_p (v) + \gL v^{p-1} = f(u) + t \gl _0 +  \gL u^{p-1}    &\qquad \text{in }\, \Omega  \\
\qquad\qquad\qquad\ \ v = \, 0 &\qquad \text{on }\, \partial\Omega, 
\end{cases}
\ee
and $K(u)= H(0,u)$, i.e. $v= K(u) $ solves the problem
\be
\begin{cases}
-\Delta_p (v) + \gL v^{p-1} = f(u) +   \gL u^{p-1}    &\qquad \text{in }\, \Omega  \\
\qquad\qquad\qquad\ \ v = \, 0 &\qquad \text{on }\, \partial\Omega, 
\end{cases}
\ee

\par
By the estimates in \cite{DiB, To, Lieb}, $K: C^0(\overline {\Omega }) \to C^{1, \beta} (\overline {\Omega })$ is a continuous operator, and therefore as an operator $K: X \to X $, $K$ is compact. Likewise,  $H: [0,1] \times C \to C $ is a compact homotopy.\par
By \eqref{Semipositivita} we have that \par
$f(s) + \Lambda s^{p-1} \geq 0 $ for every $s \in [0, + \infty )$, \par
 and by   the weak and strong  maximum principles we  deduce that if $u$ is nonnegative  in $\Omega $, then 
$K(u)$ is nonnegative as well, and in fact positive if it does not vanish in $\Omega$.
Likewise   $H(t,u) $  is positive if $u$  is positive.\par
This implies that $K:C \to C $ and $H: [0,1] \times C \to C $, where $C$ is the cone of nonnegative functions in $X$. \par
\medskip
We have to verify that, for a suitable choice of $\gl _0 $, hypotheses a), b), and c) in Theorem  \ref{Kras} are satisfied.\par
Let us observe that if $0 \neq u= H(t,u)$, then $u$ solves the problem
\be \label{casi b e c} 
\begin{cases}
-\Delta_p (u)   = f(u) + t \gl _0     &\qquad \text{in }\, \Omega  \\
\qquad\ \ \, u >\,0 &\qquad  \text{in }\, \Omega  \\
\qquad\ \ \, u = \, 0 &\qquad \text{on }\, \partial\Omega . 
\end{cases}
\ee
Likewise if $0 \neq u= sK (u) $, $0 \leq s \leq 1$,  then $u$ solves the problem
\be \label{caso a}
\begin{cases}
-\Delta_p (u)   = s^{p-1}f(u)- \gL (1-s^{p-1}) u^{p-1}   &\qquad \text{in }\, \Omega  \\
\qquad\ \ \, u >\,0 &\qquad  \text{in }\, \Omega  \\
\qquad\ \ \, u = \, 0 &\qquad \text{on }\, \partial\Omega .
\end{cases}
\ee

The property c) (for any choice of $R>0 $) follows from Lemma \ref{costantelambda0}, since $u= H(1,u) $ is the  equation \eqref{eqlemma} for $\gl = \gl _0$. \par
\medskip
Property b) follows by observing that the a priori estimates that we proved  are satisfied if we take 
$f(u) + t \lambda _0 $ instead of $f(u) $, and uniformly with respect to $t \in [0,1]$. \par
So, if $ \Vert u \Vert _{\infty } \leq C $ and we take $R = C+1$, there are not solutions of the equation
$u= H(t,u) $, i.e. \eqref{casi b e c}, with $\Vert u \Vert =R $. \par
\medskip
Finally property a) follows from Poincar\' e's inequality or equivalently from  definition of the first eigenvalue for the $p$-Laplacian. \par
In fact, by hypothesis $H_0)$, there exists $0 < \gl < \gl _1 $ and 
$ r_0  >0 $  such that $ \frac {f(s)}{s^{p-1 }}  \leq  \gl \ $ for any  $ s \in (0,r_0] $.\par
Let $0 < r \leq  r_0 $ and 
suppose that  $u \neq 0 $ solves $u= s K(u)$, i.e. \eqref{caso a} with $0 \leq s \leq 1$ and $\Vert u \Vert _{\infty } = r $.
Then, taking $u$ as a test function we obtain 
\begin{eqnarray*}
\int _{\Omega } |Du|^p \, dx \kern-9pt &=&\kern-9pt s^{p-1}\kern-2pt  \int _{\Omega } uf(u)   \, dx \kern-2pt-\kern-2pt \gL (1-s^{p-1})\kern-2pt  \int _{\Omega } u^{p} \, dx \kern-1pt\leq\kern-1pt  s^{p-1}\kern-2pt  \int _{\Omega } uf(u)   \, dx\\
&=&\kern-9pt s^{p-1} \int _{\Omega } \frac {f(u)} {u^{p-1}}\, u^p \, dx    \leq  s^{p-1} \gl   \int _{\Omega }  u^p \, dx 
\leq  \gl   \int _{\Omega }  u^p \, dx 
\end{eqnarray*}
which is a contradiction since $\gl < \gl _1 $.  \par
This implies that there are not nontrivial solutions of \eqref{caso a} with $0 \leq s \leq 1$ and $\Vert u \Vert _{\infty } = r \leq r_0 $.
\end{proof}
\par
\bigskip

\section*{Acknowledgments.}
We would like to thank Professor Paolo Baroni for many discussion and references about regularity problems.

\end{document}